\documentclass{amsart}
\usepackage{amsmath}
\usepackage{amsfonts}
\usepackage{amssymb}
\usepackage{graphicx}

 \newtheorem{theorem}{Theorem}[section]
 
 \newtheorem{lemma}[theorem]{Lemma}
 \newtheorem{proposition}[theorem]{Proposition}
 \theoremstyle{definition}
 \newtheorem{definition}[theorem]{Definition}
 \theoremstyle{remark}

 \numberwithin{equation}{section}

 \newcommand{\f}{\frac}

\begin{document}

\title[$n$-th Relative Nilpotency Degree and Relative
$n$-Isoclinism Classes]
 {$n$-th Relative Nilpotency Degree\\ and Relative
$n$-Isoclinism Classes}
\author[R. Rezaei]{Rashid Rezaei}

\address{Department of Mathematics, Faculty of
Mathematical Sciences\\
University of Malayer\\
Post Box: 657719, 95863, Malayer, Iran}
\email{ras$\_$rezaei@yahoo.com}

\author[F.G. Russo]{Francesco G. Russo}
\address{Departement of Mathematics\\
University of Palermo\\
via Archirafi 14, 90123,  Palermo, Italy} \email{francescog.russo@yahoo.com}


\date{\today}




\begin{abstract}
P. Hall introduced the notion of isoclinism between two groups more
than 60 years ago. Successively, many authors have extended such a
notion in different contexts. The present paper deals with the
notion of relative $n$-isoclinism, given by N. S. Hekster in 1986,
and with the notion of $n$-th relative nilpotency degree, recently
introduced in literature.
\end{abstract}

\maketitle

\section{Introduction and Statement of Results}
Every locally compact topological group $G$ admits a left Haar measure $\mu_G$, which is a positive Radon
measure on a $\sigma$-algebra containing Borel sets with the property that $\mu_G(xE)=\mu_G(E)$ for each element
$x$ of the measure space $G$ (see \cite[Theorem 2.8]{hm}). The support of $\mu_G$ is $G$ and it is usually
unbounded, but if $G$ is compact, then $\mu_G$ is bounded. For this reason we may assume without ambiguity that
a compact group $G$ has a unique probability measure space, where $\mu_G$ is normalized, that is, $\mu_G(G)=1$.

Let $G$ be a compact group with the normalized Haar measure $\mu_G$. On the product measure space $G \times G$,
it is possible to consider the product measure $\mu_G\times \mu_G$. If $$C_2=\{(x,y)\in G \times G  \ | \
[x,y]=1\},$$ then $C_2=f^{-1}(1)$, where $f:G \times G \rightarrow G$ is defined via $f(x,y)=[x,y]$. It is clear
that $f$ is continuous and so $C_2$ is a compact and measurable subset of $G \times G $. Therefore it is
possible to define the $commutativity$ $degree$ of $G$ as \[d(G)=(\mu_G\times \mu_G)(C_2).\]

We may extend $d(G)$ as follows. Suppose that $n\geq1$ is an
integer, $G^n$ is the product of $n$-copies of $G$  and $\mu^n_G$
that of $n$-copies of $\mu_G$. The $n$-$th$ $nilpotency$ $degree$ of
$G$ is defined to be
\[d^{(n)}(G)=\mu^{n+1}_G(C_{n+1}),\] where
\[C_{n+1}=\{(x_1,\ldots,x_{n+1})\in G^{n+1} \ | \ [x_1,x_2,...,x_{n+1}]=1
  \}.\]

Obviously, if $G$ is finite, then $G$ is a compact group with the
discrete topology and so the Haar measure of $G$ is the counting
measure. Then we have as special situation for a finite group $G$:
\[d^{(n)}(G)=\mu^{n+1}_G(C_{n+1})= \frac{|C_{n+1}|}{|G|^{n+1}}.\]
When $n=1$, literature can be found in \cite{elr, er, gal, gr, gus,
l2}.

Let $G$ be a compact group and $H$ be a closed subgroup of $G$. It
is possible to define \[D_2=\{(h,g)\in H \times G  \ | \
[h,g]=1\}.\] Then $D_2=\phi^{-1}(1)$, where $\phi:H \times G
\rightarrow G$ is defined via $\phi(h,g)=[h,g]$. It is clear that
$\phi$ is continuous and so $D_2$ is a compact and measurable subset
of $H \times G $. Note that $\phi$ is the restriction of $f$ under
$H\times G$. This remark shows why $H$ has to be required as closed
subgroup of $G$. Then we may define the $relative$ $commutativity$
$degree$ of $H$ with respect to $G$ as \[d(H,G)= (\mu_H \times
\mu_G) (D_2).\] Considering \[D_{n+1}=\{(h_1,, ..., h_n, g)\in H^n
\times G  \ | \ [h_1,h_2,...,h_n,g]=1\},\] we define  the $relative$
$n$-$th$ $nilpotency$ $degree$ of $H$ with respect to $G$ as
\[d^{(n)}(H,G)= (\mu^n_H \times \mu_G) (D_{n+1}).\]

The following notion is fundamental for stating our results (see for terminology  \cite{ b1, b2, er, hs,
 h1, hek, hm, rob}).

 \begin{definition} \label{d:1}
 Let $G_1$, $G_2$ be two groups,
$H_1$ a subgroup of $G_1$ and $H_2$  a subgroup of $G_2$. A pair
$(\alpha,\beta)$ is said
to be a relative $n$-isoclinism from $(H_1,G_1)$ to $(H_2,G_2)$ if we
 have the following conditions:\begin{enumerate}
\item[(i)]$\alpha$ is an isomorphism from $G_1/Z_n(G_1)$ to $G_2/Z_n(G_2)$ such that the restriction of $\alpha$ under
$H_1/(Z_n(G_1)\cap H_1)$ is an isomorphism from $H_1/(Z_n(G_1)\cap H_1)$ to $H_2/(Z_n(G_2)\cap H_2)$, that is,
the map $\alpha^{n+1} : \frac{H_1}{Z_n(G_1)\cap H_1}\times ...\times\frac{H_1}{Z_n(G_1)\cap H_1}\times \frac
{G_1}{Z_n(G_1)} \rightarrow \frac
 {H_2}{Z_n(G_2)\cap H_2}\times...\times\frac {H_2}{Z_n(G_2)\cap H_2}\times \frac
 {G_2}{Z_n(G_2)}$ is an isomorphism;
\item[(ii)]$\beta$ is an isomorphism from $[_nH_1,G_1]$ to $[_nH_2,G_2]$;\\
\item[(iii)]Considering  for each $h_1,...,h_n \in H_1$, $k_1,...,k_n \in H_2$, $g_1 \in
G_1$, $g_2\in G_2$ there exists a commutative diagram in which the map \[\gamma (n,H_1,G_1) : ((h_1(Z_n(G_1)\cap
H_1),...,h_n(Z_n(G_1)\cap
 H_1),g_1Z_n(G_1)))\in \frac{H_1}{Z_n(G_1)\cap H_1}\times...\]\[ ...\times\frac{H_1}{Z_n(G_1)\cap H_1}\times \frac {G_1}{Z_n(G_1)} \mapsto [h_1,
..., h_n,g_1] \in [_nH_1,G_1]\] and the map \[\gamma (n,H_2,G_2) : ((k_1(Z_n(G_2)\cap H_2),...,k_n(Z_n(G_2)\cap
H_2),g_2Z_n(G_2)))\in \frac
 {H_2}{Z_n(G_2)\cap H_2}\times...\]\[...\times\frac {H_2}{Z_n(G_2)\cap H_2}\times \frac
 {G_2}{Z_n(G_2)} \mapsto [k_1, ..., k_n,g_2]\in [_nH_2,G_2],\] can be composed by the rule
\[\gamma (n,H_2,G_2) \ \circ \  \alpha^{n+1}= \beta \ \circ \ \gamma (n,H_1,G_1).\]
\end{enumerate}
\end{definition}

It is easy to see that Definition \ref{d:1} is well posed. In
particular, two groups $G_1$ and $G_2$ in Definition \ref{d:1} are
called $isoclinic$ if $n=1$, $H_1=G_1$ and $H_2=G_2$.

P. Hall already pointed out that two groups which are isoclinic can
allow us to define the $isoclinism$ as an equivalence relation
\cite{h1}. Then it is possible to classify two groups with respect
to their $isoclinism$ $class$. See \cite{hs} as general reference on
the topic. Successively, this crucial passage was pointed out by J.
Bioch in \cite{b1,b2} and by N. S. Hekster in \cite{hek} with
respect to generalizations of the notion of isoclinism. At the same
way, Definition \ref{d:1} allows us to define an equivalence
relation, which extends that of  P. Hall in \cite{h1}, that of J.
Bioch in \cite{b1,b2} and that of N. S. Hekster in \cite{hek}. Two
groups which satisfy Definition \ref{d:1} are called $relative$
$n$-$isoclinic$ and we will write briefly $(H_1,G_1) \
_{\widetilde{n}} \ (H_2,G_2)$. It is easy to see that it is possible
to classify two groups with respect to their $relative$
$n$-$isoclinism$ $class$.

The main results of the present paper are listed below and they generalize \cite[Theorem 4.2]{er}.

\begin{theorem}\label{t:1} Assume that $G_1$ and $G_2$ are two compact groups, $H_1 $ is
a closed subgroup of $G_1$ and $H_2$ is a closed subgroup of $G_2$. If $(H_1,G_1) \ _{\widetilde{n}} \
(H_2,G_2)$, then $d^{(n)}(H_1,G_1) =d^{(n)}(H_2,G_2)$\end{theorem}

\begin{theorem}\label{t:2}  Let $G$ be a compact group and $H$ be a
closed subgroup of $G$. If $G=HZ_n(G)$, then $d^{(n)}(H)=d^{(n)}(H,G)=d^{(n)}(G).$ \end{theorem}

\section{Proof of Main Theorems and Some Consequences}
This section is devoted to describe our main results.

\begin{lemma}\label{l:1}Assume that $G$ is a compact group, $H$ is a closed subgroup of
$G$ and $C_G([h_1,...,h_n])$ is the centralizer of the commutator
$[h_1, ...,h_n]$ in $G$ for some elements $h_1, ..., h_n$ in $H$.
Then $$ d^{(n)}(H,G)=\int_H \ldots \left( \int_H \mu_G (C_G([h_1,
..., h_n]))d\mu_H (h_1)\right) \ldots d\mu_H(h_n),$$ where
$$\mu_G(C_G([h_1,...,h_n]))=\int_G\chi_{_{D_{n+1}}}(h_1,..., h_n,g)
d\mu_G(g),$$ and $\chi_{_{D_{n+1}}}$ denotes the characteristic map of the set $D_{n+1}$.\end{lemma}

\begin{proof} Since
$$\mu_G(C_G([h_1,...,h_n]))=\int_G \chi_{_{D_{n+1}}}(h_1,...,
 h_n,g)d\mu_G(g)$$ we
have by Fubini-Tonelli's Theorem:
\[\begin{array}{lcl}
d^{(n)}(H,G)&=&(\mu^n_H\times \mu_G)(D_{n+1})= \int_{H^n \times G}
 \chi_{_{D_{n+1}}}(d\mu^n_H\times d\mu_G)\vspace{.3cm}\\
&=&\int_H ... \left( \int_H \left( \int_G
 \chi_{_{D_{n+1}}}(h_1,...,h_n,g) d\mu_G(g) \right) d\mu_H (h_1)\right) ... \ d\mu_H(h_n)\vspace{.3cm}\\
&=&\int_H ... \left( \int_H \mu(C_G([h_1,...,h_n])) d\mu_H(h_1)
 \right)... \ d\mu_H(h_n).\end{array}\]
 \end{proof}

We recall to convenience of the reader \cite[Section 6, Theorem
3]{chow}. Let $(\Omega_1,\mathcal{A}_1)$ and
$(\Omega_2,\mathcal{A}_2)$ be two measurable spaces. The mapping
$X:\Omega_1\rightarrow\Omega_2$ is said to be a $measurable$
 $transformation$ from $(\Omega_1,\mathcal{A}_1)$ to
$(\Omega_2,\mathcal{A}_2)$ if
$X^{-1}(\mathcal{A}_2)\subseteq\mathcal{A}_1.$ Now for a measurable
transformation $X$ from a measure space
$(\Omega_1,\mathcal{A}_1,\mu)$ to a measurable space
$(\Omega_2,\mathcal{A}_2)$ one can check that the measure $\mu$
induces a measure $\nu$ on $\mathcal{A}_2$ via
$\nu\{A\}=\mu\{X^{-1}(A)\},$ for all $A\in\mathcal{A}_2.$ This
induced measure is denoted by $\mu X^{-1}.$ Let $\phi$ be an
isomorphism from a compact group $G_1$ with Haar measure $\mu_{G_1}$
onto a compact group $G_2$ with Haar measure $\mu_{G_2}$. It is
clear that $\phi$ is a measurable transformation from $G_1$ to
$G_2$, then by the uniqueness of the Haar measure on compact groups,
we have that $\mu_{G_2}=\mu_{G_1}\phi^{-1}$ and so
\[(*) \hspace{2cm} \int_{G_2}fd\mu_{G_2}=\int_{G_1}(f\phi)d\mu_{G_1}. \]

\begin{proof}[Proof of Theorem \ref{t:1}] Let $(\alpha,\beta)$ be a
relative $n$-isoclinism from $(H_1,G_1)$ to $(H_2,G_2)$ as in
Definition \ref{d:1}. Assume that $n=1$. Let $\mu_{H_1}$ be a
normalized Haar measure on $H_1$, $\mu_{H_2}$ on $H_2$, $\mu_{G_1}$
on $G_1$, $\mu_{G_2}$ on $G_2$, $\nu_{H_1}$ on $H_1/(H_1\cap
Z(G_1))$, $\nu_{H_2}$ on $H_2/(H_2\cap Z(G_2))$, $\nu_{G_1}$ on
$G_1/Z(G_1)$ and $\nu_{G_2}$ on $G_2/Z(G_2)$. Note that these
requirements are well posed, because $H_1$ is a closed subgroup of
$G_1$ and $H_2$ is a closed subgroup of $G_2$.

With obvious meaning of symbols, we have
\[\begin{array}{lcl}
d(H_1,G_1)&=&\int_{H_1}\int_{G_1}\chi_{_{D_2}}(h_1,g_1)d\mu_{G_1}(g_1)d\mu_{H_1}(h_1)\vspace{.3cm}\\
&=&\int_{H_1}\left(\int_{\frac{G_1}{Z(G_1)}}\left(\int_{Z(G_1)}\chi_{_{D_2}}(h_1,z_1g_1)d\mu_{Z(G_1)}(z_1)\right)d\nu_{G_1}(\bar{g_1})\right)d\mu_{H_1}(h_1)\vspace{.3cm}\\
&=&\int_{H_1}\left(\int_{\frac{G_1}{Z(G_1)}}\left(\int_{Z(G_1)}\chi_{_{D_2}}(h_1,g_1)d\mu_{Z(G_1)}(z_1)\right)d\nu_{G_1}(\bar{g_1})\right)d\mu_{H_1}(h_1)\vspace{.3cm}\\
&=&\int_{H_1}\left(\int_{\frac{G_1}{Z(G_1)}}\chi_{_{D_2}}(h_1,g_1)\left(\int_{Z(G_1)}d\mu_{Z(G_1)}(z_1)\right)d\nu_{G_1}(\bar{g_1})\right)d\mu_{H_1}(h_1)\vspace{.3cm}\\
&=&\int_{H_1}\left(\int_{\frac{G_1}{Z(G_1)}}\chi_{_{D_2}}(h_1,g_1)d\nu_{G_1}(\bar{g_1})\right)d\mu_{H_1}(h_1)\vspace{.3cm}\\
\end{array}\] now by Fubini-Tonelli's Theorem we
have
\[\begin{array}{lcl}
=\int_{\frac{G_1}{Z(G_1)}}\left(\int_{H_1}\chi_{_{D_2}}(h_1,g_1)d\mu_{H_1}(h_1)\right)d\nu_{G_1}(\bar{g_1})\vspace{.3cm}\\
\end{array}\]
then we may proceed as before:
\[\begin{array}{lcl}
=\int_{\frac{G_1}{Z(G_1)}}\left(\int_{\frac{H_1}{H_1\cap
 Z(G_1)}}\left(\int_{H_1\cap Z(G_1)}\chi_{_{D_2}}(h_1\zeta_1
 ,g_1)d\mu_{Z(G_1)\cap
 H_1}(\zeta_1)\right)d\nu_{H_1}(\bar{h_1})\right)d\nu_{G_1}(\bar{g_1})\vspace{.3cm}\\
=\int_{\frac{G_1}{Z(G_1)}}\left(\int_{\frac{H_1}{H_1\cap
 Z(G_1)}}\left(\int_{H_1\cap Z(G_1)}\chi_{_{D_2}}(h_1 ,g_1)d\mu_{Z(G_1)\cap
 H_1}(\zeta_1)\right)d\nu_{H_1}(\bar{h_1})\right)d\nu_{G_1}(\bar{g_1})\vspace{.3cm}\\
=\int_{\frac{G_1}{Z(G_1)}}\left(\int_{\frac{H_1}{H_1\cap
 Z(G_1)}}\chi_{_{D_2}}(h_1 ,g_1)\left(\int_{H_1\cap Z(G_1)}d\mu_{Z(G_1)\cap
 H_1}(\zeta_1)\right)d\nu_{H_1}(\bar{h_1})\right)d\nu_{G_1}(\bar{g_1})\vspace{.3cm}\\
=\int_{\frac{G_1}{Z(G_1)}}\left(\int_{\frac{H_1}{H_1\cap
 Z(G_1)}}\chi_{_{D_2}}(h_1
 ,g_1)d\nu_{H_1}(\bar{h_1})\right)d\nu_{G_1}(\bar{g_1})\vspace{.3cm}\\
\end{array}\]
now the commutativity of the diagram in Definition \ref{d:1} and the
fact that $\beta$ is an isomorphism of compact groups allow us to
write
\[\begin{array}{lcl}
&=&\int_{\frac{G_1}{Z(G_1)}}\left(\int_{\frac{H_1}{H_1\cap
 Z(G_1)}}\chi_{_{D_2}}(\gamma(1,H_1,G_1)(\bar{h_1}
 ,\bar{g_1}))d\nu_{H_1}(\bar{h_1})\right)d\nu_{G_1}(\bar{g_1})\vspace{.3cm}\\
&=&\int_{\frac{G_1}{Z(G_1)}}\left(\int_{\frac{H_1}{H_1\cap
 Z(G_1)}}\chi_{_{D_2}}(\beta(\gamma(1,H_1,G_1)(\bar{h_1}
 ,\bar{g_1})))d\nu_{H_1}(\bar{h_1})\right)d\nu_{G_1}(\bar{g_1})\vspace{.3cm}\\
 &=&\int_{\frac{G_1}{Z(G_1)}}\left(\int_{\frac{H_1}{H_1\cap
 Z(G_1)}}\chi_{_{D_2}}(\gamma(1,H_2,G_2)(\alpha^2(\bar{h_1}
 ,\bar{g_1})))d\nu_{H_1}(\bar{h_1})\right)d\nu_{G_1}(\bar{g_1})\vspace{.3cm}\\
\end{array}\]
now we may apply $(*)$ above so
\[\begin{array}{lcl}
&=&\int_{\frac{G_2}{Z(G_2)}}\left(\int_{\frac{H_2}{H_2\cap
 Z(G_2)}}\chi_{_{D_2}}(\gamma(1,H_2,G_2)(\bar{k_1}
 ,\bar{g_2}))d\nu_{H_2}(\bar{k_1})\right)d\nu_{G_2}(\bar{g_2})\vspace{.3cm}\\
&=&\int_{\frac{G_2}{Z(G_2)}}\left(\int_{\frac{H_2}{H_2\cap
 Z(G_2)}}\chi_{_{D_2}}(k_1
 ,g_2)d\nu_{H_2}(\bar{k_1})\right)d\nu_{G_2}(\bar{g_2})\vspace{.3cm}\\
\end{array}\]
then we may proceed as before
\[\begin{array}{lcl}
&=&\int_{\frac{G_2}{Z(G_2)}}\left(\int_{\frac{H_2}{H_2\cap
 Z(G_2)}}\chi_{_{D_2}}(k_1 ,g_2)\left(\int_{H_2\cap Z(G_2)}d\mu_{Z(G_2)\cap
 H_2}(\zeta_2)\right)d\nu_{H_2}(\bar{k_1})\right)d\nu_{G_2}(\bar{g_2})\vspace{.3cm}\\
&=&\int_{\frac{G_2}{Z(G_2)}}\left(\int_{\frac{H_2}{H_2\cap
 Z(G_2)}}\left(\int_{H_2\cap Z(G_2)}\chi_{_{D_2}}(k_1
 ,g_2)d\mu_{Z(G_2)\cap
 H_2}(\zeta_2)\right)d\nu_{H_2}(\bar{k_1})\right)d\nu_{G_2}(\bar{g_2})\vspace{.3cm}\\
&=&\int_{\frac{G_2}{Z(G_2)}}\left(\int_{\frac{H_2}{H_2\cap
 Z(G_2)}}\left(\int_{H_2\cap Z(G_2)}\chi_{_{D_2}}(k_1\zeta_2 ,g_2)d\mu_{Z(G_2)\cap
 H_2}(\zeta_2)\right)d\nu_{H_2}(\bar{h_2})\right)d\nu_{G_2}(\bar{g_2})\vspace{.3cm}\\
&=&\int_{\frac{G_2}{Z(G_2)}}\left(\int_{H_2}\chi_{_{D_2}}(k_1,g_2)d\mu_{H_2}(k_1)\right)d\nu_{G_2}(\bar{g_2})\vspace{.3cm}\\
&=&\int_{H_2}\left(\int_{\frac{G_2}{Z(G_2)}}\chi_{_{D_2}}(k_1,g_2)d\nu_{G_2}(\bar{g_2})\right)d\mu_{H_2}(k_1)\vspace{.3cm}\\
\end{array}\]\[\begin{array}{lcl}&=&\int_{H_2}\left(\int_{\frac{G_2}{Z(G_2)}}\left(\int_{Z(G_2)}\chi_{_{D_2}}(k_1,g_2)d\mu_{Z(G_2)}(z_2)\right)d\nu_{G_2}(\bar{g_2})\right)d\mu_{H_2}(k_1)\vspace{.3cm}\\
&=&\int_{H_2}\left(\int_{\frac{G_2}{Z(G_2)}}\left(\int_{Z(G_2)}\chi_{_{D_2}}(k_1,g_2z_2)d\mu_{Z(G_2)}(z_2)\right)d\nu_{G_2}(\bar{g_2})\right)d\mu_{H_2}(k_1)\vspace{.3cm}\\
&=&\int_{H_2}\int_{G_2}\chi_{_{D_2}}(k_1,g_2)d\mu_{G_2}(g_2)d\mu_{H_2}(k_1)\vspace{.3cm}\\
&=&d(H_2,G_2).\vspace{.3cm}\\
\end{array}\]
The result follows in this case.

If we assume that $n>1$, then we may repeat step by step the
previous argument, writing the corresponding integrals for the
($n+1$)-tuple $(h_1,h_2, \ldots, h_n, g_1)$ of $H^n_1\times G_1$ and
for the ($n+1$)-tuple $(k_1,k_2, \ldots, k_n,g_2)$ of $H^n_2\times
G_2$. Therefore the result follows completely.\end{proof}

The proof of Theorem \ref{t:2} is mainly due to the next two lemmas,
where the notations of Definition \ref{d:1} have been adopted. Note
that there is not hypothesis of compactness in the following two
lemmas.

\begin{lemma}\label{l:2}$(H_1,G_1) \ _{\widetilde{n}} \ (H_2,G_2)$ if and only if there
exist normal subgroups $N_1\leq Z_n(G_1)$, $N_2\leq Z_n(G_2)$ and two isomorphisms $\alpha$ and $\beta$ such
that $\alpha:G_1/N_1\rightarrow G_2/N_2,$ $\beta:[_nH_1,G_1]\rightarrow[_nH_2,G_2],$ $\alpha(H_1/(N_1\cap
H_1))=H_2/(N_2\cap H_2)$  for all  $g_1\in G_1$  and $h_i\in H_1,$ $\beta([h_1,...,h_n,g_1])=[k_1,...,k_n,g_2],$
where $g_2\in\alpha(g_1N_1)$, $k_i\in\alpha(h_i(N_1\cap H_1))$ and $1\leq i\leq n$.\end{lemma}

\begin{proof} Assume that $(H_1,G_1) \ _{\widetilde{n}} \ (H_2,G_2)$. We
may write $N_1=Z_n(G_1)$ and $N_2=Z_n(G_2)$. From this, the result
follows.

Conversely, one can show easily that
$\alpha(Z_n(G_1)/N_1)=Z_n(G_2)/N_2$, therefore by third isomorphism
theorem, $\alpha$ induces the isomorphism $\alpha'$ from
$G_1/Z_n(G_1)$ to $G_2/Z_n(G_2)$ by the rule
$\alpha'(g_1Z_n(G_1))=g_2Z_n(G_2)$, where $g_2\in\alpha(g_1N_1)$.
Furthermore $$\alpha'(h_1(Z_n(G_1)\cap H_1))=h_2(Z_n(G_2)\cap H_2)$$
in which $h_2\in\alpha(h_1(N_1\cap H_1))$. Hence $(\alpha',\beta)$
is a relative $n$-isoclinism from $(H_1,G_1)$ to
$(H_2,G_2)$.\end{proof}

\begin{lemma}\label{l:3} Let $H$ be a subgroup of a group $G=HZ_n(G)$. Then
$$(H,H) \ _{\widetilde{n}} \ (H,G) \ _{\widetilde{n}} \
(G,G).$$\end{lemma}

\begin{proof} We want to show that $(H,H) \ _{\widetilde{n}} \ (H,G)$. Let
$G=HZ_n(G)$. We may easily see that $Z_n(H)=Z_n(G)\cap H$. Thus
$H/Z_n(H)= H/(Z_n(G)\cap H)$ is isomorphic to
$HZ_n(G)/Z_n(G)=G/Z_n(G)$.

Therefore $\alpha:H/Z_n(H)\rightarrow G/Z_n(G)$ is an isomorphism
which is induced by the inclusion $i:H\rightarrow G.$ Furthermore,
we can consider $\alpha$ as isomorphism from $H/Z_n(H)$ to
$H/(Z_n(G)\cap H)$.

On another hand, $[_nH,G]=[_nH,HZ_n(G)]=\gamma_{n+1}(H)$. By Lemma
\ref{l:2}, the pair $(\alpha,1_{\gamma_{n+1}(H)})$ allows us to
state
 that $(H,H) \ _{\widetilde{n}} \ (H,G)$.

The remaining cases $(H,H) \ _{\widetilde{n}} \ (G,G)$ and $(H,H) \
_{\widetilde{n}} \ (H,G)$ follow by a similar argument. \end{proof}

\begin{proof}[Proof of Theorem \ref{t:2}] It follows from Theorem
\ref{t:1} and Lemma \ref{l:3}. \end{proof}

We end with some consequences of the main results. We recall the
 following lemma, which is used in many proofs of \cite{elr, gal, gr, gus}.

\begin{lemma}[\textrm{See} \cite{er}, \textrm{Lemma 3.1}] \label{l:4}Let $k$ be a positive integer. If $H$ is a closed subgroup of a
compact group $G$, then
\[\mu_G(H)=\left\{\begin{array}{lcl} \frac{1}{k},&\,\,& |G:H|=k\\
0,&& |G:H|=\infty.\end{array}\right.\]\end{lemma}


\begin{proposition} \label{l:5} Let $G$ be a compact group and $G/Z(G)$ be a $p$-group of order
$p^k$, where $p$ is a prime and $k\geq 2$ is an integer. Then
$d(G)\leq \frac{p^k + p-1}{p^{k+1}}$ and the equality holds if $G$
is isoclinic to an extra-special $p$-group of order $p^{k+1}$.
\end{proposition}

\begin{proof} By Lemma \ref{l:4} we have $\mu_G(Z(G))=\f{1}{p^k}$ and
$\mu_G(C_G(x))\leq \f1p$ for all $x\notin Z(G)$. Therefore
\[\begin{array}{lcl}
d(G)&=&\int_G\mu_G(C_G(x))d\mu_G(x)=\int_{Z(G)}\mu_G(C_G(x))+\int_{G-Z(G)}\mu_G(C_G(x))\vspace{.3cm}\\
&\leq&\f1p+(\f{p-1}{p})\mu_G(Z(G))=\f1p+(\f{p-1}{p})\f1{p^k}=\f{p^k+p-1}{p^{k+1}}.
\end{array}\]

By Theorem \ref{t:1}, there is no loss of generality in assuming
that $G$ is an extra-special $p$-group of order $p^{k+1}$. For all
$x\not\in Z(G)$ we may consider the epimorphism
 $\varphi_x:G\rightarrow G'$ defined by $\varphi_x(y)=[x,y]$ in which the kernel is $C_G(x)$. Therefore $\mu_G(C_G(x))=\f1p$ and
 so the above inequality becomes $d(G)=\f{p^k + p-1}{p^{k+1}}$.
 \end{proof}

An easy consequence has been shown below.

\begin{proposition}\label{p:1}If $G$ is a compact group in which $Z(G)$ has infinite index in
$G$, then $d(G)\leq
 \frac{1}{2}$. In particular, if $G$ is isoclinic to an extra-special $p$-group for some prime $p$, then $d(G)=
 \frac{1}{p}$.\end{proposition}

\begin{proof} By Lemma \ref{l:4} we have $\mu_G(Z(G))=0$ and
$\mu_G(C_G(x))\leq \f12$ for all $x\not\in Z(G)$. From this, we have
\[\begin{array}{lcl}
d(G)=\int_G\mu_G(C_G(x))d\mu_G(x)\leq\frac{1}{2}+\frac{1}{2}\mu_G(Z(G))=\frac{1}{2}.
\end{array}\]

Following the same argument,  if $G/Z(G)$ is an infinite $p$-group, then  $d(G)\leq \frac{1}{p}$. In particular,
if $G$ is isoclinic to an extra-special $p$-group, then we may argue as in the last part of the proof of
Proposition \ref{l:5}, concluding that $\mu_G(C_G(x))=\f1p$ for every $x\not \in Z(G)$. Therefore the equality
$d(G)= \frac{1}{p}$ holds and the result follows. \end{proof}

\end{document}